\documentclass[12pt,reqno]{article}

\usepackage{amsmath,amssymb,amsthm}
\usepackage{amsfonts}
\usepackage{amscd}
\usepackage{graphicx}
\usepackage{xcolor}
\usepackage{float}
\usepackage{booktabs}

\usepackage{mathtools}
\usepackage[toc,page]{appendix}
\usepackage{microtype}
\usepackage{tikz}
\usepackage[colorlinks=true,linkcolor=blue,citecolor=blue,urlcolor=blue]{hyperref}

\def\modd#1 #2{#1\ \mbox{\rm (mod}\ #2\mbox{\rm )}}

\makeatletter
\renewcommand*{\@fnsymbol}[1]{\ensuremath{\ifcase#1\or *\or \ddagger\or \S\or \P\or \|\or **\or \dagger\dagger \or \ddagger\ddagger \else\@ctrerr\fi}}
\makeatother

\usepackage{fullpage}

\usepackage{latexsym}
\usepackage{epsf}
\usepackage{breakurl}

\newcommand{\Z}{\mathbb Z}
\newcommand{\N}{\mathbb N}

\newcommand{\vp}{\nu_3}
\newcommand{\vTwo}{\nu_2}

\newcommand{\Q}{\mathbb{Q}}

\begin{document}

\theoremstyle{plain}
\newtheorem{theorem}{Theorem}
\newtheorem{corollary}[theorem]{Corollary}
\newtheorem{lemma}[theorem]{Lemma}
\newtheorem{proposition}[theorem]{Proposition}

\theoremstyle{definition}
\newtheorem{definition}[theorem]{Definition}
\newtheorem{example}[theorem]{Example}
\newtheorem{conjecture}[theorem]{Conjecture}

\theoremstyle{remark}
\newtheorem{remark}[theorem]{Remark}

\author{Stijn Cambie\thanks{Department of Computer Science, KU Leuven Campus Kulak-Kortrijk, 8500 Kortrijk, Belgium. Supported by a postdoctoral fellowship by the Research Foundation Flanders (FWO) with grant number 1225224N. E-mail: \href{mailto:stijn.cambie@hotmail.com}{\tt stijn.cambie@hotmail.com}.} \and Erik Kalviainen\thanks{Waterloo, Ontario, Canada. E-mail: \href{mailto:ekalvi@gmail.com}{\tt ekalvi@gmail.com}.}  \and J. Shallit\thanks{
School of Computer Science,
University of Waterloo,
Waterloo, ON  N2L 3G1 
Canada. E-mail: 
\href{mailto:shallit@uwaterloo.ca}{\tt shallit@uwaterloo.ca}. Research supported by NSERC Grant RGPIN 2024-03725.}}

\title{Brown-Gerver-Ramsey Theorems in Small Dimensions}

\maketitle

\begin{abstract}
We consider infinite walks in $\N^k$ with standard unit basis vector steps that avoid
$t$ collinear points, and show that these walks exist for $(k,t) \in \{(6,3), (4,4), (3,7)\}$.  In particular, our construction for $k = 3$ improves the previous bound $189$, obtained by Lidbetter, to $7$.  Our results also imply the existence of infinite words over small finite alphabets that are weakly abelian squarefree (resp., weakly abelian cubefree, weakly abelian 6th-power-free).
\end{abstract}

\section{Introduction}

In 1971 Tom C. Brown observed that every infinite walk on the lattice points of the plane, using only the steps $(0,1)$ and $(1,0)$, must contain $n$ collinear points for all $n \geq 2$.  He submitted his result to the Advanced Problem section of the {\it American Mathematical Monthly} \cite{Bro71b}, and a nice solution was published in 1972 by Peter Montgomery \cite{Mo72}.

Of course, it is easy to create such a walk for which every line intersects it at only finitely many points:  for example the following walk resembling a parabola, set $P_0 = (0,0)$, and
$$ P_{n+1} = \begin{cases}
    P_n + (0,1), & \text{if $P_n = (k^2,k-1)$ for some integer $k$}; \\
    P_n + (1,0), & \text{otherwise.} 
    \end{cases}
$$

Brown's observation was generalized by Gerver and Ramsey to the case where the steps are chosen from an arbitrary fixed finite set $S \subseteq \N^2$ \cite{GR79}.
Gerver and Ramsey also found an upper bound on the walk length (defined as number of steps) guaranteeing the existence of $n$ collinear points, as a function of the size of $S$ and the norm of the largest step. Gerver \cite{Ger79} proved a superpolynomial lower bound on the walk length.  These upper and lower bounds were recently improved by
Korsky \cite{Ko26a}.

Gerver and Ramsey addressed the same question in $\N^3$.  For that version of the problem, they proved there is an infinite walk, using only the
unit vector steps 
$(1,0,0), (0,1,0), (0,0,1)$, with no
$5^{11} + 1 = 48828126$ collinear points.
(This number was recently reduced to 189
by Lidbetter \cite{Lid24}.)
They also raised the question of what dimension $k$ is required to avoid
three collinear points with a finite set of steps.

Recently this latter question was solved by the first two authors \cite{CK26}.  Their preprint proved the existence of an infinite walk in $\Z^3$ using a set of 16 steps. 

In this paper we return to the original question of Brown and restrict our attention to walks in $\N^k$ using only the {\it standard unit basis vectors}, that is, those of the form $${\bf e}_{k,i} = (\, \overbrace{0,\ldots, 0}^i, 1, 
\overbrace{0,\ldots, 0}^{k-i-1} \,)$$
for $0 \leq i < k$.  By the points of a walk we mean its visited lattice points.  The results in \cite{CK26} immediately show that there is an infinite walk in $\N^{16}$, using only the standard unit basis vectors, with no $3$ collinear points \cite{Sha26}, but we improve this here.

We now summarize the main results of this paper.  All these walks use only the standard unit basis vectors as steps.

\begin{itemize}
\item Theorem~\ref{thm:main}: there is an infinite walk on $\N^6$
with no $3$ collinear points.  In particular this supersedes the recent paper of Korsky \cite{Ko26b}.

\item Theorem~\ref{thm:d4}:  there is an infinite walk on $\N^4$ with no $4$ collinear points.

\item Theorem~\ref{thm:d3}:  there is an infinite walk on $\N^3$ with no $7$ collinear points.
\end{itemize}

The proofs are presented in Sections~\ref{sec:6} and~\ref{sec:34}.
Readers less interested in combinatorics of words (equivalence explained in Section~\ref{sec2}), can skip Corollary~\ref{cor:dim6wordformulation} in Section~\ref{sec:6}
and Section~\ref{sec5}.

We do not know if these results can be improved for unit basis vector steps.  For example, we cannot currently rule out an infinite walk on $\N^3$ with
no $4$ collinear points.  We have
constructed a finite walk of length $32768$ with this property.   However, as Brown observed \cite{Bro71a}, the longest
walk on $\N^3$ with no $3$ collinear points is of length $7$.

We also cannot rule out an infinite walk
in $\N^4$ with no $3$ collinear points.  We have constructed a finite walk of length $5500$ with this property.

Throughout, $\N=\{0,1,2,\ldots\}$, and word and vertex indices start
at zero.

Our constructions use $p$-adic valuations for the primes $p =2$ and $p = 3$.  For a nonzero integer $n$, we let $\nu_p (n)$ be the exponent of the highest power of $p$ dividing $n$. This is extended to nonzero rational numbers $m/n$ by
$\nu_p(m/n) = \nu_p(m) - \nu_p (n)$.

\section{Weak abelian powers}
\label{sec2}

Our results have a natural interpretation in combinatorics on words.  Let $\Sigma_k = \{0,1,\ldots, k-1\}$ be an alphabet of cardinality $k$.  For a word $w$, let
$|w|_a$ be the number of occurrences of
the letter $a$ in $w$.  For
$w \in \Sigma_k^*$ define
its Parikh vector
$\psi(w)$ to be $(|w|_0, |w|_1, \ldots,
|w|_{k-1})$.
We say that
a word $w$ is a {\it weak abelian $t$'th power} if $w = x_1 x_2 \cdots x_t$, where
each word $x_i$ is nonempty and
$${{\psi(x_1)} \over {|x_1|}} = 
{{\psi(x_2)} \over {|x_2|}} = \cdots =
{{\psi(x_t)} \over {|x_t|}}.$$
We say a word $w$ {\it avoids\/} weak abelian $t$'th powers if no nonempty factor is a weak abelian $t$'th power.  See, for example, \cite{AP16,FP23}.

\begin{proposition}
Let $w=w_0w_1\cdots w_{n-1}\in\Sigma_k^n$. Define
\[
P_i=\sum_{0 \leq j < i}{\bf e}_{k,w_j}
\]
for $0 \leq i \leq n$.
Then the walk $ P_0,\ldots,P_n$ contains no $t+1$ 
collinear points if and only if $w$ avoids weak abelian $t$'th powers.
\label{prop1}
\end{proposition}

\begin{proof}
If there are $t+1$ collinear points, their
consecutive displacements are positive multiples of a common vector;
division by their coordinate sums gives the same normalized Parikh
vector for all $t$ intervening nonempty blocks. Conversely, such $t$
blocks give $t+1$ collinear boundary points.
\end{proof}

Thus our results can be restated as
\begin{itemize}
\item There is an infinite word over $\Sigma_6$
with no weak abelian squares.  

\item There is an infinite word over $\Sigma_4$ with no weak abelian cubes.

\item There is an infinite word over $\Sigma_3$ with no weak abelian $6$th powers.
\end{itemize}
We also produce explicit constructions for these three words, as the image of a fixed point of a morphism.

\begin{remark}
Given a finite word $w=w_0w_1\cdots w_{n-1}$, we can determine
the smallest exponent $p\geq2$ such that $w$ avoids weak abelian
$p$'th powers in $O(n^2\log n)$ time and $O(n)$ space. For each $0\leq j<n$, compute
\[
 \frac{\psi(w_jw_{j+1}\cdots w_{\ell-1})}{\ell-j}
 \qquad (j<\ell\leq n),
\]
sort these vectors lexicographically, separately for each $j$, and find the maximum multiplicity
$M$ of an identical vector. Then $p=M+1$.
This can be improved to $O(n^2)$ time
and $O(n)$ space using
radix sort with the appropriate buckets.
\end{remark}

\section{Dimension 6}\label{sec:6}

\begin{theorem}
There is an infinite walk in $\N^6$, using only the standard unit basis vectors as steps, with no $3$ collinear points.
\label{thm:main}
\end{theorem}

We begin with a walk on the Gaussian integers, adapted from the construction in
\cite{CK26}.  Let $i = \sqrt{-1}$, and define $\sigma_n\in\{0,1,2,3\}$ recursively by $\sigma_0=0$ and
\begin{equation}\label{eq:states}
 \sigma_{2n}\equiv \modd{-\sigma_n} {4},
 \qquad
 \sigma_{2n+1}\equiv \modd{1-\sigma_n} {4}.
\end{equation}
For $n \geq 0$, define
\begin{equation}\label{eq:gaussian}
 u_n=i^{\sigma_n},
 \qquad
 z_n=\sum_{0\le r<n}u_r\in\Z[i].
\end{equation}
Thus $(z_n)$ is a nearest-neighbor walk in the Gaussian lattice.  For
$\varepsilon\in\{0,1\}$, Eq.~\eqref{eq:states} gives
\begin{equation}\label{eq:dyadic}
u_{2n+\varepsilon}=i^\varepsilon\,\overline{u_n},
 \qquad
 z_{2n+\varepsilon}
   =(1+i)\, \overline{z_n}+\varepsilon\, \overline{u_n},
\end{equation}
where $\overline{x}$ is the complex conjugate.
Indeed, the second identity follows by pairing
$u_{2r}+u_{2r+1}=(1+i)\overline{u_r}$.

\begin{lemma}\label{lem:chords}
If $0\leq m<n$ and $\sigma_m=\sigma_n$, then $z_m\ne z_n$ and
\begin{equation}\label{eq:chord-valuation}
 \vTwo\bigl(|z_n-z_m|^2\bigr)=\vTwo(n-m).
\end{equation}
\end{lemma}

\begin{proof}
Let $t=\vTwo(n-m)$ and set $(m_0,n_0)=(m,n)$.  We recursively
define pairs $(m_j,n_j)$ satisfying $n_j-m_j=(n-m)/2^j$ for $0\le j\le t$. Suppose that $j<t$.
Then $n_j-m_j$ is even, so $m_j$ and $n_j$ have the same parity.
Thus, for a unique $\varepsilon_j\in\{0,1\}$, we may write
\[
 m_j=2m_{j+1}+\varepsilon_j,
 \qquad
 n_j=2n_{j+1}+\varepsilon_j.
\]
The two maps
\[
 r\longmapsto-r,\qquad r\longmapsto1-r
\]
are injective on $\mathbb Z/4\mathbb Z$.  Hence the recursion~\eqref{eq:states}, together with
$\sigma_{m_j}=\sigma_{n_j}$, implies
$
 \sigma_{m_{j+1}}=\sigma_{n_{j+1}}
$.
In particular, $u_{m_{j+1}}=u_{n_{j+1}}$.  Applying Eq.~\eqref{eq:dyadic} with $n=m_{j+1}$ and
$n=n_{j+1}$, we get
\[
 z_{n_j}-z_{m_j}
   =(1+i)(\, \overline{z_{n_{j+1}}-z_{m_{j+1}}}\, ).
\]
Consequently,
\[
 |z_{n_j}-z_{m_j}|^2
   =2|z_{n_{j+1}}-z_{m_{j+1}}|^2.
\]
Iterating for $j=0,\ldots,t-1$ yields
\[
 |z_n-z_m|^2
   =2^t|z_{n_t}-z_{m_t}|^2.
\]

By construction, $n_t-m_t$ is odd.  Moreover,
\[
 z_{n_t}-z_{m_t}
   =\sum_{r=m_t}^{n_t-1}u_r
\]
is a sum of an odd number of elements of $\{1,i,-1,-i\}$.  Write this
sum as $x+iy$,  
where
$x+y$ is odd. Thus exactly one of $x$ and $y$ is odd, so
\[
 |z_{n_t}-z_{m_t}|^2=x^2+y^2
\]
is odd and, in particular, nonzero.  It follows that $z_n-z_m\ne0$ and
$ \vTwo\bigl(|z_n-z_m|^2\bigr)
   =t
   =\vTwo(n-m)$.
\end{proof}

\begin{proof}[Proof of Theorem~\ref{thm:main}]
Let 
$\delta_n = (\sigma_{n+1}-\sigma_n) \bmod4$.  From Eq.~\eqref{eq:states},
\begin{equation}\label{eq:increments}
 \delta_{2n}=1,
 \qquad
 \delta_{2n+1}\equiv-1-\delta_n\pmod4.
\end{equation}
Since $\delta_0=1$, induction gives $\delta_n\in\{1,2\}$ for every $n$.
Consequently, the only possible ordered pairs $(\sigma_n, \sigma_{n+1})$ are
\begin{equation}
 (0,1),\ (1,2),\ (2,3),\ (3,0),\
 (0,2),\ (1,3),\ (2,0),\ (3,1).
 \label{pairs}
\end{equation}

We now define two functions on $\{0,1,2,3\}$ by
\[
 (d_0,d_1,d_2,d_3)=(0,1,i,1+i),
 \qquad
 (c_0,c_1,c_2,c_3)=(0,0,0,1-i),
\]
and define the auxiliary points
\[
 R_n=(d_{\sigma_n},z_n+c_{\sigma_n},n).
\]
If $(\sigma_n,\sigma_{n+1})=(r,s)$, then
\begin{equation}\label{eq:shadow-step}
 R_{n+1}-R_n
 =\bigl(d_s-d_r,\ i^r+c_s-c_r,\ 1\bigr).
\end{equation}
The values \(d_r\) place the four possible values of \(\sigma_n\) at the extreme points of a square, while the offsets \(c_r\) make the possible increments of \(R_n\) collapse from eight ordered pairs to the six vectors
 shown in Figure~\ref{fig:steps}.
\begin{figure}[htb]
\centering
\begin{minipage}[c]{.29\linewidth}
\centering
\begin{tikzpicture}[x=1.25cm,y=1.05cm,every node/.style={font=\small}]
  \draw[gray!70,line width=.6pt] (0,0) rectangle (1,1);
  \fill (0,0) circle (1.5pt) node[below left] {$d_0=0$};
  \fill (1,0) circle (1.5pt) node[below right] {$d_1=1$};
  \fill (0,1) circle (1.5pt) node[above left] {$d_2=i$};
  \fill (1,1) circle (1.5pt) node[above right] {$d_3=1+i$};
\end{tikzpicture}
\end{minipage}\hfill
\begin{minipage}[c]{.67\linewidth}
\centering
\begin{tabular}{c@{\qquad}c}
\toprule
$(\sigma_n, \sigma_{n+1})$ & $R_{n+1}-R_n$ \\
\midrule
$01$       & $(1,\ 1,\ 1)$ \\
$12$       & $(-1+i,\ i,\ 1)$ \\
$23$       & $(1,\ -i,\ 1)$ \\
$30$       & $(-1-i,\ -1,\ 1)$ \\
$02,13$    & $(i,\ 1,\ 1)$ \\
$20,31$    & $(-i,\ -1,\ 1)$ \\
\bottomrule
\end{tabular}
\end{minipage}
\caption{Collapse to six steps.}
\label{fig:steps}
\end{figure}

We claim that no three points $R_n$ are collinear.  Suppose otherwise, and
choose $a<b<c$ such that $R_a,R_b,R_c$ are collinear.  Write
$p=b-a$ and $q=c-b$.  Since the last coordinate of $R_n$ is $n$, the affine
parameter is fixed:
\begin{equation}\label{eq:slopes}
 \frac{R_b-R_a}{p}
 =\frac{R_c-R_b}{q}
 =\frac{R_c-R_a}{p+q}.
\end{equation}
In the first complex coordinate this says
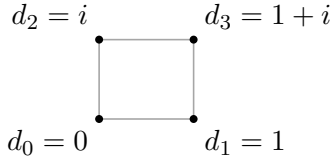
\begin{equation}\label{eq:convexity}
 d_{\sigma_b}
 =\frac{q}{p+q}d_{\sigma_a}
  +\frac{p}{p+q}d_{\sigma_c}.
\end{equation}
The four numbers $d_0,d_1,d_2,d_3$ are the vertices of a square. Since all four are extreme points of their convex hull, the strict convex combination in
\eqref{eq:convexity} forces
\begin{equation}\label{eq:equal-states}
 \sigma_a=\sigma_b=\sigma_c.
\end{equation}

The correction terms $c_{\sigma_n}$ now cancel from the second complex
coordinates in~\eqref{eq:slopes}.  Hence
\begin{equation}\label{eq:gaussian-slopes}
 \frac{z_b-z_a}{p}
 =\frac{z_c-z_b}{q}
 =\frac{z_c-z_a}{p+q}.
\end{equation}
By Lemma~\ref{lem:chords}, for example, we get the square-modulus valuations
\[
 \vTwo\!\left(\left|\frac{z_b-z_a}{p}\right|^2\right)
 =\vTwo\bigl(|z_b-z_a|^2\bigr)-2\vTwo(p)=-\vTwo(p).
\]
Thus the three nonzero fractions in~\eqref{eq:gaussian-slopes} have
squared-modulus valuations
\begin{equation}\label{eq:slope-valuations}
 -\vTwo(p),\qquad -\vTwo(q),\qquad -\vTwo(p+q),
\end{equation}
respectively.  Equality in~\eqref{eq:gaussian-slopes} would therefore give
\(
 \vTwo(p)=\vTwo(q)=\vTwo(p+q).
\)
This is impossible: if $p$ and $q$ have the same $2$-adic valuation, then
$p+q$ has a larger one.  Thus no three points of the auxiliary walk $(R_n)$ are collinear.

Finally, list the six possible increments of $R_n$ in Figure~\ref{fig:steps} as
$v_1,\ldots,v_6$.  Starting from $P_0=0$, replace every occurrence of $v_j$
by the standard basis vector ${\bf e}_{6, j-1}\in\N^6$.  The real-linear map
\[
 T:\mathbb R^6\longrightarrow\mathbb C^2\times\mathbb R,
 \qquad T({\bf e}_{j-1})=v_j,
\]
satisfies $T(P_n)=R_n$ for every $n$.  Moreover, the sum of the coordinates of $P_n$ is $n$, so the
points $P_n$ are distinct.  A collinear triple among
the $P_n$ would map to three distinct collinear points among the $R_n$, which
we have just proved impossible.  This completes the construction.
\end{proof}

\begin{corollary}\label{cor:dim6wordformulation}
The steps of the walk are given by
the unit basis vectors
$({\bf e}_{6,q_i})_{i \geq 0}$,
where $q_i$ is the image under
$\tau$ of the fixed point of the
morphism $\gamma$, defined as follows:
\begin{table}[H]
\centering
\begin{tabular}{c|c|c}
$j$ & $\gamma(j)$ & $\tau(j)$ \\
\hline
0 & 01 & 0 \\
1 & 20 & 1 \\
2 & 34 & 2 \\
3 & 25 & 3 \\
4 & 62 & 4 \\
5 & 03 & 1 \\
6 & 67 & 5 \\
7 & 36 & 4 \\
\end{tabular}
\end{table}
Thus $$\tau(\gamma^\omega(0)) = 01203401215201203403543401203401 \cdots$$ is weakly abelian squarefree.
\end{corollary}

\begin{proof}
Let $\mathbf t=t_0t_1t_2\cdots=\gamma^\omega(0)$.
Associate each letter with an ordered pair  by
\[
\begin{array}{c|cccccccc}
j&0&1&2&3&4&5&6&7\\ \hline
\pi(j)&(0,1)&(1,3)&(3,0)&(1,2)&(2,0)&(0,2)&(2,3)&(3,1).
\end{array}
\]
We claim that
\[
\pi(t_n)=(\sigma_n,\sigma_{n+1}) \qquad(n\geq0).
\]
Indeed, if $\pi(j)=(r,s)$ and $\gamma(j)=j_0j_1$, direct
inspection of the table gives
\[
\pi(j_0)=(-r,1-r),\qquad
\pi(j_1)=(1-r,-s),
\]
where all entries are reduced modulo $4$. On the other hand,
the recurrence~\eqref{eq:states} gives
\begin{align*}
(\sigma_{2n},\sigma_{2n+1})
  &=(-\sigma_n,1-\sigma_n),\\
(\sigma_{2n+1},\sigma_{2n+2})
  &=(1-\sigma_n,-\sigma_{n+1}),
\end{align*}
again modulo $4$. Since
$\pi(t_0)=\pi(0)=(0,1)=(\sigma_0,\sigma_1)$ and
$t_{2n}t_{2n+1}=\gamma(t_n)$, the claim follows by induction.

Now, if $\pi(t_n)=(r,s)$, the corresponding increment of the
auxiliary walk is
\[
R_{n+1}-R_n=(d_s-d_r,\ i^r+c_s-c_r,\ 1).
\]
By Figure~\ref{fig:steps}, the pairs associated with letters
$1$ and $5$, namely $(1,3)$ and $(0,2)$, give the same
increment. Similarly, the pairs associated with letters $4$
and $7$, namely $(2,0)$ and $(3,1)$, give the same increment.
These are the only identifications among the eight pairs.
They are exactly the identifications made by $\tau$, since
\[
\tau(1)=\tau(5)=1,\qquad
\tau(4)=\tau(7)=4,
\]
while the remaining four letters receive the distinct labels
$0,2,3,5$.

Thus $q_n=\tau(t_n)$ labels precisely the increment
$R_{n+1}-R_n$. Replacing these six increment types by
${\bf e}_{6,0},\ldots,{\bf e}_{6,5}$ therefore gives the walk
constructed in the proof of Theorem~\ref{thm:main}, up to a
permutation of coordinates. This walk has no three collinear
points, so its step word $\tau(\gamma^\omega(0))$ is weakly
abelian squarefree.
\end{proof}

\section{Dimensions $4$ and $3$}\label{sec:34}

The constructions here are similar to those of the previous section,
except that we now use base-$3$ representations and the field norm
\[
 N_{\Q(\sqrt3)/\Q}(X+Y\sqrt3)=X^2-3Y^2.
\]

Write $n=\sum_{j\ge0}f_j 3^j$ in base $3$ and set
\[
 a_n=(-1)^{f_1+f_3+\cdots},\qquad
 b_n=(-1)^{f_0+f_2+\cdots},\qquad
 R_n=(X_n,Y_n,n),
\]
where $X_n=\sum_{j<n}a_j$ and $Y_n=\sum_{j<n}b_j$. Appending a ternary
digit gives, for $r=0,1,2$,
\begin{equation}\label{eq:rec}
 a_{3n+r}=b_n,\qquad b_{3n+r}=(-1)^r a_n,
\end{equation}
and therefore
\begin{equation}\label{eq:sums}
 X_{3n+r}=3Y_n+r b_n,\qquad
 Y_{3n+r}=X_n+\varepsilon_r a_n,
 \quad(\varepsilon_0,\varepsilon_1,\varepsilon_2)=(0,1,0).
\end{equation}

Now we state and prove an analogue of Lemma~\ref{lem:chords}.

\begin{lemma}\label{lem:ternary-chords}
If $0\leq m<n$ and $(a_m,b_m)=(a_n,b_n)$, then
$(X_n-X_m)^2-3(Y_n-Y_m)^2$ is nonzero and
\begin{equation}\label{eq:val}
 \vp\!\left((X_n-X_m)^2-3(Y_n-Y_m)^2\right)=\vp(n-m).
\end{equation}
\end{lemma}

\begin{proof}
Let $t=\vp(n-m)$ and set $(m_0,n_0)=(m,n)$.  We recursively
define pairs $(m_j,n_j)$ for $0\leq j\leq t$.  Suppose that $j<t$.
Then $n_j-m_j=(n-m)/3^j$ is divisible by $3$, so $m_j$ and $n_j$ have the
same remainder modulo $3$.  Thus, for a unique $r_j\in\{0,1,2\}$,
we may write
\[
 m_j=3m_{j+1}+r_j,
 \qquad
 n_j=3n_{j+1}+r_j.
\]
For each $r\in\{0,1,2\}$, the map
\[
 (a,b)\longmapsto (b,(-1)^r a)
\]
is injective on $\{1,-1\}^2$.  Hence the recursion~\eqref{eq:rec},
together with $(a_{m_j},b_{m_j})=(a_{n_j},b_{n_j})$, implies
\[
 (a_{m_{j+1}},b_{m_{j+1}})
   =(a_{n_{j+1}},b_{n_{j+1}}).
\]
Applying Eq.~\eqref{eq:sums} with $n=m_{j+1}$ and
$n=n_{j+1}$, in both cases with $r=r_j$, and subtracting,
the correction terms cancel, giving
\[
 X_{n_j}-X_{m_j}=3(Y_{n_{j+1}}-Y_{m_{j+1}}),
 \qquad
 Y_{n_j}-Y_{m_j}=X_{n_{j+1}}-X_{m_{j+1}}.
\]
For $0\leq j\leq t$, put
\[
 D_j=(X_{n_j}-X_{m_j})^2-3(Y_{n_j}-Y_{m_j})^2.
\]
The preceding identities give $D_j=-3D_{j+1}$ for $j<t$.
Iterating yields
\[
 D_0=(-3)^tD_t.
\]

By construction, $n_t-m_t$ is not divisible by $3$.  Write
$m_t=3u+r$ and $n_t=3v+s$, where $r,s\in\{0,1,2\}$ and $r\ne s$.
Since $a_{m_t}=a_{n_t}$,  Eq.~\eqref{eq:rec} gives $b_u=b_v$.
It follows from~\eqref{eq:sums} that
\[
 X_{n_t}-X_{m_t}
   \equiv s b_v-r b_u
   =(s-r)b_u\not\equiv0\pmod3.
\]
Consequently,
\[
 D_t\equiv (X_{n_t}-X_{m_t})^2\equiv1\pmod3.
\]
Thus $D_t$ is nonzero and not divisible by $3$.  Since
$D_0=(-3)^tD_t$, we conclude that $D_0\ne0$ and
\[
 \vp\!\left((X_n-X_m)^2-3(Y_n-Y_m)^2\right)
   =\vp(D_0)=t=\vp(n-m).
\]
\end{proof}

\begin{lemma}\label{lem:core}
Let $E=\{n\ge0:(a_n,b_n)=(1,1)\}$.  No four points $R_n$ with $n\in E$
are collinear.
\end{lemma}

\begin{proof}
Suppose $R_{n_0},\ldots,R_{n_3}$, with $n_0<\cdots<n_3$ in $E$, lie
on one line.  There are rational slopes $\alpha,\beta$ such that every
chord has the form
\[
 (X_{n_j}-X_{n_i},Y_{n_j}-Y_{n_i})
   =(n_j-n_i)(\alpha,\beta).
\]
Put $c=\alpha^2-3\beta^2$.  Lemma~\ref{lem:ternary-chords} implies
that $c\ne0$ and, for every $i<j$, gives
\[
 \vp(n_j-n_i)=\vp(c)+2\vp(n_j-n_i)
\]
by Eq.~\eqref{eq:val}.  Hence all six differences have one common
valuation $e=-\vp(c)\geq0$.  The four integers
\[
 \frac{n_i-n_0}{3^e}\qquad(0\leq i\leq3)
\]
would therefore have pairwise distinct residues modulo $3$, which is
impossible.
\end{proof}

Enumerate $E$ as
\[
 0=N_0<N_1<N_2<\cdots.
\]

\begin{proposition}\label{prop:return-steps}
For every $k\geq 0$,
\[
 N_{k+1}-N_k\in\{2,4,6,8\}.
\]
More precisely, if
\[
 r=\frac{N_{k+1}-N_k}{2},
\]
then
\begin{equation}\label{eq:return}
 \frac12\bigl(R_{N_{k+1}}-R_{N_k}\bigr)=d_r,
\end{equation}
where $d_r$ is defined by the 
following table:
\begin{equation}\label{eq:d_r_vectors}
 \begin{array}{c|cccc}
  r&1&2&3&4\\ \hline
  d_r&(1,0,1)&(-1,1,2)&(0,-1,3)&(-2,0,4)
 \end{array} \ .
\end{equation}
\end{proposition}

We present two proofs for this proposition, so readers can choose the proof best suited to their background.
\begin{proof}[First proof.]
If $n=9q+3u+v$, with $u,v\in\{0,1,2\}$, then by~\eqref{eq:rec} (applied twice)
\[
 (a_n,b_n)=((-1)^ua_q,(-1)^vb_q).
\]
The residues modulo $9$ giving state $(1,1)$ given $(a_q,b_q)$ are respectively
\[
\begin{array}{c|cccc}
(a_q,b_q)&(1,1)&(1,-1)&(-1,1)&(-1,-1)\\ \hline
n-9q&0,2,6,8&1,7&3,5&4
\end{array}.
\]
Notice that $\left(\{(1,1),(1,-1),(-1,1),(-1,-1)\}, \cdot \right)$, where $\cdot$ is the entrywise multiplication forms the Klein-four group, i.e., every element is its inverse and so the table can also be read reversely; starting from $(a_0,b_0)=(1,1)$ one obtains the latter $(a_{n-9q},b_{n-9q})$.

When $q$ is increased by one, a ternary carry changes the parity of exactly
one of the two alternating digit sums (notice that $n=\sum_{j\ge0}f_j3^j  \equiv \sum_{j\ge0} f_j \pmod 2$).
Noticing that the corresponding residue $n-9q$ changes parity as well, reading consecutive entries in the
above table therefore shows that
$N_{k+1}-N_k\in\{2,4,6,8\}.$

In each of the cases, the sum of the $(a_j,b_j)$ depends solely on the difference and one can check~\eqref{eq:d_r_vectors}.

For this calculation, one sums the corresponding $(a_q,b_q),$ multiplied with the base state (to correspond with $(1,1)$).
For example, if one has $N_i \equiv 5 \pmod 9$ and $N_{i+1} \equiv 4 \pmod 9$,
we have 
\begin{align*}
&(-1,1)\cdot \left( (-1,1)+(1,1)+(1,-1)+(1,1) \right)+ (-1,-1)\cdot \left((1,1)+(1,-1)+(1,1)+(-1,1) \right)\\=&(-1,1)\cdot(2,2)+(-1,-1)\cdot(2,2)=(-4,0), \mbox{ which is twice } (-2,0). \end{align*}
The thirteen cases, computed as above, can be summarized as follows

\begin{table}[ht]
\centering
\renewcommand{\arraystretch}{1.2}
\begin{tabular}{ccl}
\hline
$N_{k+1}-N_k$
&
$(N_k\bmod 9)\longrightarrow (N_{k+1}\bmod 9)$
&
$\frac12\bigl(R_{N_{k+1}}-R_{N_k}\bigr)$
\\
\hline
$2$
&
$0\to2,\;3\to5,\;6\to8,\;7\to0,\;8\to1$
&
$(1,0,1)=d_1$
\\
$4$
&
$2\to6,\;5\to0,\;8\to3$
&
$(-1,1,2)=d_2$
\\
$6$
&
$1\to7,\;4\to1,\;7\to4$
&
$(0,-1,3)=d_3$
\\
$8$
&
$4\to3,\;5\to4$
&
$(-2,0,4)=d_4$
\\
\hline
\end{tabular}
\caption{The thirteen possible transitions between consecutive
elements of \(E\). An arrow \(r\to s\) records
\(N_k\equiv r\pmod 9\) and \(N_{k+1}\equiv s\pmod 9\).}
\label{tab:return-transitions}
\end{table} \qedhere
\end{proof}

\begin{proof}[Second proof.]
Set
\[
 A=(1,1),\qquad B=(1,-1),\qquad
 C=(-1,1),\qquad D=(-1,-1).
\]
The recurrence
\[
 a_{3n+s}=b_n,\qquad b_{3n+s}=(-1)^s a_n
 \qquad(s\in\{0,1,2\})
\]
induces the three-uniform morphism
\begin{equation}\label{eq:mu}
 \mu(A)=ABA,\qquad \mu(B)=CDC,\qquad
 \mu(C)=BAB,\qquad \mu(D)=DCD.
\end{equation}
Thus the word
\[
 \mathbf x=((a_n,b_n))_{n\geq0},
\]
written over the alphabet $\{A,B,C,D\}$, is the fixed point
$\mu^\omega(A)$.  The occurrences of $A$ in $\mathbf x$ are precisely
the positions $N_0,N_1,N_2,\ldots$.

Define morphisms $\varphi$ and $\eta$ by
\begin{equation}\label{eq:return-morphisms}
\begin{array}{c|c|l|c}
 j&\varphi(j)&\eta(j)&|\eta(j)|/2\\ \hline
 0&01&AB&1\\
 1&0020&ACDC&2\\
 2&034&ABDCDB&3\\
 3&00224&ACDCDCDB&4\\
 4&03220&ABDCDCDC&4.
\end{array}
\end{equation}
Let $\mathbf y=\varphi^\omega(0)$.  Direct substitution gives
\begin{align*}
 \mu(\eta(0))&=\eta(0)\eta(1),\\
 \mu(\eta(1))&=\eta(0)\eta(0)\eta(2)\eta(0),\\
 \mu(\eta(2))&=\eta(0)\eta(3)\eta(4),\\
 \mu(\eta(3))&=\eta(0)\eta(0)\eta(2)\eta(2)\eta(4),\\
 \mu(\eta(4))&=\eta(0)\eta(3)\eta(2)\eta(2)\eta(0).
\end{align*}
Equivalently,
\[
 \mu\circ\eta=\eta\circ\varphi.
\]
Since $\varphi(\mathbf y)=\mathbf y$, it follows that
\[
 \mu(\eta(\mathbf y))
   =\eta(\varphi(\mathbf y))
   =\eta(\mathbf y).
\]
The word $\eta(\mathbf y)$ begins with $A$, and $\mu$ has a unique fixed
point beginning with $A$.  Hence
\[
 \eta(\mathbf y)=\mathbf x.
\]

Every word $\eta(j)$ begins with $A$ and contains no other occurrence of
$A$.  Therefore the boundaries of the consecutive $\eta$-blocks in
$\eta(\mathbf y)=\mathbf x$ are exactly the positions $N_k$.  In
particular, for every $k\geq0$,
\begin{equation}\label{eq:return-block}
 \mathbf x[N_k..N_{k+1}-1]=\eta(y_k),
 \qquad
 N_{k+1}-N_k=|\eta(y_k)|.
\end{equation}
The length column of~\eqref{eq:return-morphisms} now gives
$N_{k+1}-N_k\in\{2,4,6,8\}$.

It remains to compute the corresponding displacement.  Regard each letter
$L\in\{A,B,C,D\}$ as its associated vector in $\{\pm1\}^2$.  From the
definition of $R_n$ and~\eqref{eq:return-block},
\[
 R_{N_{k+1}}-R_{N_k}
 =\left(
    \sum_{L\text{ in }\eta(y_k)}L,
    |\eta(y_k)|
  \right),
\]
where letters are counted with multiplicity.  The five possible blocks give
\begin{equation}\label{eq:displacements}
\begin{array}{c|c|c}
 j&\displaystyle\sum_{L\text{ in }\eta(j)}L
  &\displaystyle\frac12\left(
       \sum_{L\text{ in }\eta(j)}L,\ |\eta(j)|
     \right)\\ \hline
 0&(2,0)&(1,0,1)=d_1\\
 1&(-2,2)&(-1,1,2)=d_2\\
 2&(0,-2)&(0,-1,3)=d_3\\
 3&(-4,0)&(-2,0,4)=d_4\\
 4&(-4,0)&(-2,0,4)=d_4.
\end{array}
\end{equation}
Together with~\eqref{eq:return-block}, this proves
\eqref{eq:return}.
\end{proof}

\begin{theorem}
There is an infinite standard-basis walk in $\N^4$ with no $4$ collinear
points.
\label{thm:d4}
\end{theorem}

\begin{proof}
At the $k$th step use ${\bf e}_{4,r-1}$, where $r=(N_{k+1}-N_k)/2$.  The linear map
sending ${\bf e}_{4,r-1}$ to $d_r$ sends the $k$th walk vertex to $R_{N_k}/2$.
Thus four collinear walk points would have four collinear images; the
images are distinct because their third coordinates are the distinct
numbers $N_k/2$.  Lemma~\ref{lem:core} gives the contradiction.
\end{proof}

\begin{theorem}\label{thm:d3}
There is an infinite standard-basis walk in $\N^3$ with no seven collinear
points.
\end{theorem}
\begin{proof}
The vectors $d_1,d_2,d_3$ in \eqref{eq:d_r_vectors} are independent
(determinant $6$), and $d_1+d_4=d_2+d_3$. Hence an invertible linear map
may send
\[
 d_1\mapsto {\bf e}_{3,0},\qquad d_2\mapsto {\bf e}_{3,1},\qquad
 d_3\mapsto {\bf e}_{3,0}+{\bf e}_{3,2},
\]
and then necessarily $d_4\mapsto {\bf e}_{3,1}+{\bf e}_{3,2}$.  The images
$S_k$ of $R_{N_k}/2$ therefore form a walk with these four macro-steps,
and Lemma~\ref{lem:core} says that no four $S_k$ are collinear.

Subdivide ${\bf e}_{3,0}+{\bf e}_{3,2}$ as ${\bf e}_{3,2},{\bf e}_{3,0}$ and ${\bf e}_{3,1}+{\bf e}_{3,2}$ as ${\bf e}_{3,2},{\bf e}_{3,1}$.  Every
vertex of the resulting unit-step walk lies in
\[
 \mathcal S\cup(\mathcal S+{\bf e}_{3,2}),\qquad \mathcal S=\{S_k:k\ge0\}.
\]
Each line meets either translate in at most three points, hence contains
at most six points altogether.
\end{proof}

\begin{remark}
The bound six is attained by this construction.  Its vertices at times
$64,70,82,88,100,106$ are
\[
(38,13,13),(41,15,14),(47,19,16),(50,21,17),(56,25,19),(59,27,20),
\]
and lie on a line of primitive direction $(3,2,1)$. Thus Theorem~\ref{thm:d3}
cannot be strengthened to ``at most five'' by merely sharpening its last
estimate.\\
Also in Theorem~\ref{thm:d4} the ``at most three'' is sharp with the given construction, since the second, third and fourth point $P_2=(1,1,0,0),
P_3=(2,1,0,0) \mbox { and }
P_4=(3,1,0,0)$ are collinear.
\end{remark}

\section{A common morphism for Theorems~\ref{thm:d4} and \ref{thm:d3}}\label{sec5}

Use the same $\varphi$ and $\mathbf y$ as above and define
\begin{equation}\label{eq:codings}
\begin{array}{c|c|c|c}
 j&\varphi(j)&\tau_4(j)&\tau_3(j)\\ \hline
 0&01&0&0\\
 1&0020&1&1\\
 2&034&2&20\\
 3&00224&3&21\\
 4&03220&3&21
\end{array}
\end{equation}
Here $\tau_4$ is a letter-to-letter coding; $\tau_3$ is a non-erasing
morphism, and $20$ and $21$ denote two-letter words.

\begin{corollary}
The word $\mathbf w_4=\tau_4(\varphi^\omega(0))$ is the step word of
the four-dimensional walk in Theorem~\ref{thm:d4}, using zero-based direction
labels. In particular, it is weakly abelian cubefree.
\end{corollary}

\begin{proof}
Equation~\eqref{eq:return-block}, together with the last column
of~\eqref{eq:return-morphisms}, gives
\[
 (\mathbf w_4)_k=\frac{N_{k+1}-N_k}{2}-1.
\]
Thus $\mathbf w_4$ specifies precisely the steps of Theorem~\ref{thm:d4}, with
$d_r$ represented by ${\bf e}_{4,r-1}$. Equivalently, the linear map
${\bf e}_{4,r-1}\mapsto d_r$ sends the $k$th walk vertex to
$R_{N_k}/2$. Lemma~\ref{lem:core} rules out four collinear images; these images
are distinct because their third coordinates are $N_k/2$.
Hence the walk has no four collinear points, and Proposition~\ref{prop1}
proves the assertion about weak abelian cubes.
\end{proof}

\begin{corollary}
The word $\mathbf w_3=\tau_3(\varphi^\omega(0))$ is the step word of
the three-dimensional walk in Theorem~\ref{thm:d3}. In particular, it avoids
weak abelian sixth powers.
\end{corollary}

\begin{proof}
Theorem~\ref{thm:d3} replaces the return steps $d_1,d_2,d_3,d_4$ by the unit-step
words $0,1,20,21$, respectively. By Eq.~\eqref{eq:displacements}, the
return letters $0,1,2,3,4$ therefore give exactly the images
$0,1,20,21,21$ defining $\tau_3$. Consequently
$\mathbf w_3=\tau_3(\mathbf y)$ specifies the stated construction.

For completeness, let $S_k$ be the image of $R_{N_k}/2$ under the
invertible linear map in Theorem~\ref{thm:d3}. Lemma~\ref{lem:core} implies that each line
contains at most three points of $\mathcal S=\{S_k:k\geq0\}$.
Every extra vertex introduced by a two-step image is $S_k+{\bf e}_{3,2}$.
Thus all points lie in
$\mathcal S\cup(\mathcal S+{\bf e}_{3,2})$, and every line contains at most
six points. Proposition~\ref{prop1} then rules out weak abelian sixth powers.
\end{proof}

Alternatively we can define the word in Theorem~\ref{thm:d3} as follows:
\begin{equation}\label{eq:eight}
\begin{array}{c|c|c}
 L&\theta(L)&\pi(L)\\ \hline
 a&ab&0\\
 b&aacda&1\\
 c&a&2\\
 d&efgh&0\\
 e&a&2\\
 f&acdcdgh&1\\
 g&a&2\\
 h&efcdcda&1
\end{array}
\qquad
 \mathbf w_3=\pi(\theta^\omega(a)).
\end{equation}
To verify this, define
\[
 H(0)=a,\quad H(1)=b,\quad H(2)=cd,\quad
 H(3)=ef,\quad H(4)=gh.
\]
The table satisfies $\theta \circ  H=H \circ \varphi$ and $\pi\circ H=\tau_3$.
It follows that $H(\mathbf y)$ is a fixed point of $\theta$ starting
with $a$, hence $H(\mathbf y)=\theta^\omega(a)$, and therefore
\[
 \pi(\theta^\omega(a))=\pi(H(\mathbf y))
   =\tau_3(\mathbf y)=\mathbf w_3.
\]
The limit exists because $\theta(a)=ab$ and the morphism is
non-erasing, with $\lvert\theta^n(a)\rvert$ tending to infinity.

The construction of $\varphi$ is based on return words.
For a general statement and proof, see Durand~\cite[Proposition~19]{Dur98}:
the derived sequence at a nonempty prefix of the fixed point of a primitive morphism
is itself fixed by a primitive morphism. The identity
$\mu\circ\eta=\eta\circ\varphi$ above is a direct certificate in this instance.
Proposition~9 of the same paper gives the general conversion from a
non-erasing morphic image of a primitive fixed point to a coding of
a primitive fixed point; the eight-letter presentation above makes
that conversion explicit here. 

\section{Final words}

As we mentioned, we do not know if the result of Theorem~\ref{thm:main} could be improved to $\N^5$ or $\N^4$.  The third author found a possible construction in $\N^5$ 
\cite{Sha26},
but we have not been able to prove that it works.  Namely, consider the cyclic
morphism $\alpha(i)$ defined by adding $i$ modulo $5$
to every letter of $\alpha(0)=01213101314310$.  Then $\alpha^\omega (0)$ appears
to avoid weak abelian squares, at least as far as we have been able to check (600,000 terms), and if correct, this would give a walk in $\N^5$ with no $3$ collinear points, by the translation described in Section~\ref{sec2}.

A website devoted to explaining the results of this paper is available at\\
\centerline{\url{https://basis-walk.q5m.ai/}\, ,} and a github repository with code is\\
\centerline{\url{https://github.com/ekalvi/basis-walk}\, .}

\section*{Declaration of AI usage}

Initial proofs for the results in this paper were obtained with GPT-6 Astra. Everything was checked and rewritten by the authors.


\begin{thebibliography}{99}

\bibitem{AP16}
S. Avgustinovich and S. Puzynina.
\newblock Weak abelian periodicity of infinite words.
\newblock {\it Theory Comput. Syst.} \textbf{59} (2016), 161--179.

\bibitem{Bro71a}
T. C. Brown.
\newblock Is there a sequence on four symbols in which no two adjacent
segments are permutations of one another?
\newblock {\it Amer. Math. Monthly} {\bf 78} (1971), 886--888.

\bibitem{Bro71b}
T. C. Brown.
\newblock Advanced problem 5811.
\newblock {\it Amer. Math. Monthly} {\bf 78} (1971), 798.

\bibitem{CK26}
S. Cambie and E. Kalviainen.
\newblock An infinite small-step $\Z^3$-walk with no collinear triple.
\newblock Arxiv preprint arXiv:2609.01766 [math.CO],
September 1 2026.  Available at
\url{https://arxiv.org/abs/2609.01766v1}.

\bibitem{Dur98}
F. Durand.
\newblock A characterization of substitutive sequences using return words.
\newblock \emph{Discrete Math.} \textbf{179} (1998), 89--101.

\bibitem{FP23}
G. Fici and S. Puzynina.
\newblock Abelian combinatorics on words: a survey.
\newblock {\it Computer Sci. Review} \textbf{47} (2023), article~100532.

\bibitem{Ger79}
J. L. Gerver.
\newblock Long walks in the plane with few collinear points.
\newblock {\it Pacific J. Math.} {\bf 83} (1979), 349--355.

\bibitem{GR79}
J. L.~Gerver and L.~T. Ramsey.
\newblock On certain sequences of lattice points.
\newblock {\it Pacific J. Math.} \textbf{83} (1979), 357--363.

\bibitem{Ko26a}
S. Korsky.
\newblock North-east lattice paths with few collinear vertices.
\newblock Arxiv preprint arXiv:2607.02832 [math.CO], July 2 2026.
\newblock Available at \url{https://arxiv.org/abs/2607.02832}.

\bibitem{Ko26b}
S. Korsky.
\newblock Long lattice paths with no three collinear vertices.
\newblock Arxiv preprint arXiv:2608.07906 [math.CO], August 8 2026.
\newblock Available at \url{https://arxiv.org/abs/2608.07906}.

\bibitem{Lid24}
T. F.~Lidbetter.
\newblock Improved bound for the Gerver--Ramsey collinearity problem.
\newblock {\it Discrete Math.} \textbf{347} (2024), Article~113718.

\bibitem{Mo72}
P. L. Montgomery. 
\newblock Solution to Advanced Problem 5811: 
collinear points on a monotonic polygon.
\newblock {\it Amer. Math. Monthly} \textbf{79} (1972), 1143--1144.

\bibitem{Sha26}
J. Shallit. An infinite walk in \(\mathbb N^{16}\), using only unit steps, with no three collinear points.
\newblock ArXiv preprint arXiv:2609.05780 [math.CO] (2026).  
\newblock Available at
\url{https://arxiv.org/abs/2609.05780}.






\end{thebibliography}
\end{document}